\documentclass[12pt,a4paper,oneside]{amsart}
\pdfoutput=1
\usepackage{amssymb}
\usepackage{esint}
\usepackage{mathtools}
\usepackage{hyperref}

\usepackage[english]{babel}
\usepackage[latin1]{inputenc}
\usepackage[pdftex]{graphicx}

\theoremstyle{plain}
\newtheorem{theorem}{Theorem}
\newtheorem{lemma}[theorem]{Lemma}
\newtheorem{corollary}[theorem]{Corollary}

\theoremstyle{definition}

\theoremstyle{remark}

\newcommand{\T}{\mathbb{T}}
\newcommand{\R}{\mathbb{R}}
\newcommand{\Z}{\mathbb{Z}}
\newcommand{\C}{\mathbb{C}}
\newcommand{\N}{\mathbb{N}}
\newcommand{\Q}{\mathbb{Q}}

\newcommand{\ip}[2]{\left\langle#1,#2\right\rangle}
\newcommand{\iip}[2]{\left(#1,#2\right)}

\newcommand{\der}{\mathrm{d}}

\renewcommand{\phi}{\varphi}
\newcommand{\abs}[1]{\left| #1 \right|}
\newcommand{\aabs}[1]{\left\| #1 \right\|}

\DeclareMathOperator{\vdiv}{\overset{\scriptstyle v}{\operatorname{div}}}
\DeclareMathOperator{\hdiv}{\overset{\scriptstyle h}{\operatorname{div}}}
\DeclareMathOperator{\vgrad}{\overset{\scriptstyle v}{\nabla}}
\DeclareMathOperator{\hgrad}{\overset{\scriptstyle h}{\nabla}}

\newcommand{\schw}{\mathcal S}
\newcommand{\F}{\mathcal F}

\title{X-ray transforms in pseudo-Riemannian geometry}
\author{Joonas Ilmavirta}
\thanks{Department of Mathematics and Statistics, University of Jyv\"askyl\"a}
\address{Department of Mathematics and Statistics, University of Jyv\"askyl\"a, P.O. Box 35 (MaD) FI-40014 University of Jyv\"askyl\"a, Finland}
\email{joonas.ilmavirta@jyu.fi}
\date{\today}

\begin{document}

\begin{abstract}
We study the problem of recovering a function on a pseudo-Riemannian manifold from its integrals over all null geodesics in three geometries: pseudo-Riemannian products of Riemannian manifolds, Minkowski spaces and tori.
We give proofs of uniqueness anc characterize non-uniqueness in different settings.
Reconstruction is sometimes possible if the signature $(n_1,n_2)$ satisfies $n_1\geq1$ and $n_2\geq2$ or vice versa and always when $n_1,n_2\geq2$.
The proofs are based on a Pestov identity adapted to null geodesics (product manifolds) and Fourier analysis (other geometries).
The problem in a Minkowski space of any signature is a special case of recovering a function in a Euclidean space from its integrals over all lines with any given set of admissible directions, and we describe sets of lines for which this is possible.
Characterizing the kernel of the null geodesic ray transform on tori reduces to solvability of certain Diophantine systems.
\end{abstract}



\maketitle

\section{Introduction}

We study the problem of recovering a function on a pseudo-Riemannian manifold from its integrals over all null geodesics.
We do this in three different settings:
\begin{itemize}
\item
Pseudo-Riemannian products of compact Riemannian manifolds.
The signature $(n_1,n_2)$ is assumed to satisfy $n_1\geq2$ and $n_2\geq2$.
See section~\ref{sec:product} for details.
\item
The Minkowski space~$\R^{n_1,n_2}$ of any signature, $n_1\geq1$ and $n_2\geq2$.
If $n_1=1$ and $n_2\geq2$ (or vice versa), injectivity depends on whether we assume compact support or not.
See section~\ref{sec:minkowski} for details.
\item
The torus $\T^{n_1,n_2}=\R^{n_1,n_2}/\Z^{n_1+n_2}$ with the pseudo-Riemannian metric inherited from the Minkowski space.
Again we assume only $n_1\geq1$ and $n_2\geq2$.
The manifold has no boundary, and we only include integrals over closed null geodesics.
See section~\ref{sec:torus} for details.
\end{itemize}

A Riemannian metric on a manifold~$M$ is a vector bundle isomorphism $g\colon TM\to T^*M$ for which $g(v,v)\coloneqq\ip{gv}{v}>0$ for all $v\in TM\setminus0$ and the quadratic form corresponding to~$g$ is symemtric on each fiber~$T_xM$.
If we drop the positivity condition but insist that~$g$ be symmetric an isomorphism, we obtain a pseudo-Riemannian metric.
Each tangent space~$T_xM$ can be written as a direct sum of two subspaces so that the quadratic form given by the metric is positive definite on one and negative definite on the other.
If the dimensions of these subspaces are~$n_1$ and~$n_2$, respectively, the signature of the metric is $(n_1,n_2)$.
The numbers~$n_1$ and~$n_2$ are assumed constant, and they are automatically constant on connected manifolds.
Disconnected manifolds are not interesting for ray transforms.

If $n_2\geq1$, the pseudo-Riemannian metric tensor~$g$ does not give rise to a a distance function.
Geodesics can be defined, however, in the usual way using the Levi-Civita connection, but geodesics no longer minimize a convex functional.
The tangent vectors $v\in TM$ satisfying $g(v,v)=0$ are called null or lightlike vectors.
Geodesics with null tangent vectors are called null geodesics.

A simple example of a pseudo-Riemannian manifold of signature $(n_1,n_2)$ is the Euclidean space~$\R^{n_1+n_2}$ with the metric tensor given by the constant (diagonal) matrix
\begin{equation}
\begin{pmatrix}
I_{n_1}&0\\
0&-I_{n_2}
\end{pmatrix}
.
\end{equation}
We call this space the Minkowski space of signature $(n_1,n_2)$ and denote it by~$\R^{n_1,n_2}$.

In some settings we allow $n_1=1$ or $n_2=1$, but the results are weaker.
Riemannian and Euclidean geometry correspond to $n_2=0$, and the case $n_2=1$ (with $n_1\geq1$) is referred to as Lorentzian geometry.
If $n_1\geq2$ and $n_2\geq2$, all the ray transforms studied here are injective.
There is an unexpected gap: Assuming $n_1\geq2$, the cases $n_2=0$ and $n_2\geq2$ give relatively straightforward injectivity results, whereas the case $n_2=1$ we have some examples of non-injectivity.
We believe this is related to the fact that the (co)normal bundle of all non-zero null vectors is the whole (co)tangent bundle if and only if neither of~$n_1$ and~$n_2$ is~$1$.
See section~\ref{sec:minkowski} for more details on the normal bundle.

To make the analogy between $n_2=0$ and $n_2\geq2$ completely honest, one should restrict to null geodesics in both cases.
But $n_2=0$ corresponds to Riemannian geometry where no non-constant geodesic is null.
Therefore the mentioned gap is not entirely real.

Null geodesics have the peculiar property that a conformal change in the metric only causes a reparametrization --- as unparametrized curves they are conformally invariant.
Knowing the integrals of a function~$f$ over null geodesics of a pseudo-Riemannian metric~$cg$, with~$c$ a positive smooth function, is equivalent with knowing the integrals of~$\sqrt{c}f$ with respect to the metric~$g$.
Therefore injectivity results are invariant under conformal transformations.

\subsection{Methods and results}

We assume the signature to satisfy $n_1\geq1$ and $n_2\geq1$ unless otherwise mentioned.
For more precise statements of the results, see the theorems and corollaries as referred here.

Theorem~\ref{thm:product} states that the null geodesic X-ray transform is injective on a conformal pseudo-Riemannian product of two Riemannian manifolds~$M_1$ and~$M_2$, assuming that both manifolds are compact and have strictly convex boundary and non-positive sectional curvature.
Besides being injective on scalar functions, the ray transform is solenoidally injective on one-forms.
This means that a smooth one-form integrates to zero over all null geodesics if and only if it is the differential of a scalar function vanishing at the boundary.
The product manifold $M=M_1\times M_2$ does not have smooth boundary.
For technical convenience, we assume the functions and one-forms to vanish near the non-smooth parts of the boundary.

The proof is based on a Pestov identity, analogous to the ones used previously on Riemannian manifolds of any dimension~\cite{PSU:anosov-mfld}.

Theorem~\ref{thm:x-ray-schwartz} concerns the problem of reconstructing a Schwartz function $f\colon\R^n\to\C$ from its integrals over all lines that have direction in a given set $D\subset S^{n-1}$.
This reconstruction is possible if and only if the normal bundle of~$D$ is dense in~$\R^n$.
If~$D$ is closed, this amounts to requiring that every vector in~$\R^n$ is normal to some direction in~$D$.
If reconstruction is impossible, we characterize the kernel of the corresponding integral transform.
The proof relies on basic Fourier analysis.

The result changes dramatically if one replaces Schwartz functions with smooth functions of compact support.
Corollary~\ref{cor:x-ray-compact} states that for injectivity it suffices that~$D$ contains an arc.

The null X-ray transform on Minkowski spaces can be studied by taking~$D$ to be the unit vectors of the light cone.
Corollary~\ref{cor:minkowski} states that the null X-ray transform on the Minkowski space~$\R^{n_1,n_2}$ is injective on Schwartz functions if and only if $n_1\geq2$ and $n_2\geq2$.
On compactly supported smooth functions we have injectivity if and only if $n_1\geq2$ or $n_2\geq2$.

Theorem~\ref{thm:torus} states that a function or a distribution on~$\T^{n_1,n_2}$ can be recovered from its integrals over all closed null geodesics if and only if $n_1\geq2$ and $n_2\geq2$.
If these conditions are not met, theorem~\ref{thm:torus-ker} characterizes the kernel.
These results are based on Fourier analytic tools for X-ray transforms on tori developed in~\cite{I:torus}.
In addition, characterization of the kernel in the cases where there is one ($n_1=1$ or $n_2=1$) requires a study of solvabilty of Diophantine systems.
We solve the Diophantine problem in signature $(n,1)$ with $n\in\{1,2,3\}$, but in higher dimensions we do not reach an equally elegant characterization of the kernel of the periodic null X-ray transform.
The Diophantine results are presented in section~\ref{sec:diophantos}.

We give three completely independent proofs of the injectivity of the null X-ray transform on Minkowski the space~$\R^{n_1,n_2}$ when $n_1,n_2\geq2$.
These proofs follow simply from each one of theorem~\ref{thm:product}, corollary~\ref{cor:minkowski} and theorem~\ref{thm:torus}.

\subsection{Related results}

There is a vast literature on the Riemannian version of this problem.
For recent results in recovering a function or a tensor field from its integrals along geodesics, we refer to the survey~\cite{PSU:tensor-survey}.

The problem of recovering a function $f\colon\R^n\to\R$ from its integrals over lines with a given set~$D$ of possible directions is of practical interest in X-ray imaging.
If a cone of directions is given, the problem is known as the limited angle tomography problem.
For an in-depth overview of this topic, we refer to~\cite{F:limited-angle-thesis}.
Some other results concerning the recovery of a function in~$\R^3$ from certain restricted sets of lines can be found in~\cite{BQ:line-complex,B:partial-data-xrt}.

Our interest is not in stability or practical reconstruction methods, but mere uniqueness results.
In particular, we make no assumptions whatsoever on the set~$D$.
We are not aware of previous results for arbitrary sets of admissible directions.
Our interest in such generalized limited angle tomography problems comes from null geodesics in Minkowski spaces.
We have included the result for general~$D$ as it is of little added effort.

Null X-ray transforms have received some attention in Lorentzian geometry~\cite{LOSU:cosmic-strings,S:light-ray-support,S:scattering89,G:21cosmology}, but we are not aware of any previous work for other pseudo-Riemannian manifolds.
The Lorentzian null ray transform appears in connection to hyperbolic inverse boundary value problems~\cite{S:scattering89,RS:wave-eq,RR:hyperbolic,B:wave-eq,S:time-dep-potential,W:t-dep-potential,M:hyperbolic-dn,A:time-dep-coeff,KO:time-dep-coeff,SY:lorentz-dn}, similarly to how the Riemannian X-ray transform appears with elliptic ones. 
The transform has also applications in cosmology~\cite{LOSU:cosmic-strings,G:21cosmology}.

The Riemannian X-ray transform has proven useful in the study of inverse boundary value problems for elliptic equations, and the Lorentzian one for hyperbolic ones.
Similarly, we expect the null X-ray transform in pseudo-Riemannian geometry to be related to partial differential operators like the pseudo-Riemannian Laplace--Beltrami operator.
On a pseudo-Riemannian product of two manifolds $M_1$ and $M_2$ (see below), this operator is simply $\Delta_1-\Delta_2$, the difference of the Riemannian Laplace--Beltrami operators.

\section{The X-ray transform on pseudo-Riemannian products of Riemannian manifolds}
\label{sec:product}

Let $(M_i,g_i)$, $i=1,2$, be two compact Riemannian manifolds with boundary.
Let $M=M_1\times M_2$ and equip this product with the pseudo-Riemannian metric $g_1-g_2$; with plus sign this would be the standard Riemannian metric on the product.
The manifold~$M$ is Lorentzian if one of the dimensions is one.
We will also consider conformal multiples of this metric.
However, whenever we integrate over~$M$, we use the volume measure of the Riemannian metric $g_1+g_2$.
The product~$M$ is a manifold with corners.

We denote the ``edges'' of~$M$ by $E\coloneqq M_1\times\partial M_2\cup\partial M_1\times M_2$.
This is the irregular part of the boundary of~$M$.
We say that a function or a one-form is supported outside the edges if it vanishes in a neighborhood of~$E$.

Null geodesics on~$M$ are products of geodesics on~$M_1$ and~$M_2$ with equal speed.
Therefore the null geodesic flow lives on the light cone bundle
\begin{equation}
LM=SM_1\times SM_2\subset TM,
\end{equation}
where~$SM_i$ is the unit sphere bundle of~$M_i$.
Each fiber~$L_xM$, $x=(x_1,x_2)\in M$, is a product of two spheres: $L_xM=S_{x_1}M\times S_{x_2}M$.

\begin{theorem}
\label{thm:product}
Suppose the compact Riemannian manifolds $(M_1,g_1)$ and $(M_2,g_2)$ both have dimension two or higher, nonpositive sectional curvature and strictly convex boundary.
Equip $M=M_1\times M_2$ with the pseudo-Riemannian metric $c(x)(g_1-g_2)$, where $c\in C^\infty(M)$ is a positive conformal factor.
\begin{enumerate}
\item
If a smooth function $f\colon M\to\R$ supported outside the edges integrates to zero over all null geodesics, then $f=0$.
\item
If a smooth one-form~$\alpha$ supported outside the edges on~$M$ integrates to zero over all null geodesics, then $\alpha=\der h$ for some $h\in C^\infty(M)$ vanishing at the boundary.
\end{enumerate}
\end{theorem}

If we allowed the supports of~$f$ and~$\alpha$ meet~$E$, then the function~$u$ defined in the proof below would not necessarily be smooth.

\subsection{Vector fields and commutators}

The geodesic flow on the Riemannian unit sphere bundle~$SM_i$ is generated by the geodesic vector field~$X_i$.
We can naturally promote vector fields on~$SM_i$ to vector fields on $LM=SM_1\times SM_2$.

The geodesic vector field on~$LM$ is $X=X_1+X_2$.
This vector field generates the null geodesic flow.

In addition to the geodesic vector fields~$X_i$ we need vertical and horizontal gradients and divergences ($\hgrad_i,\hdiv_i,\vgrad_i,\vdiv_i$).
For an introduction to these, see~\cite[Section~2]{PSU:anosov-mfld}.
The situation is technically simpler in dimension two (in our setting $n_1=n_2=2$); see eg.~\cite{PSU:tensor-surface}.

In addition to these differential operators, we will use the curvature operator~$R_i$ on each~$SM_i$.
If~$M_i$ has non-positive sectional curvature, then $\ip{R_i(x,v)w}{w}\leq0$ for all $(x,v)\in SM_i$ and $w\in T_xM_i$ with $\ip{v}{w}=0$.

These operators satisfy the following commutator formulas:
\begin{equation}
\label{eq:commutators}
\begin{split}
[X_i,\vgrad_i]&=-\hgrad_i\\
[X_i,\hgrad_i]&=R_i\vgrad_i\\
\hdiv_i\vgrad_i-\vdiv_i\hgrad_i&=(n_i-1)X_i\\
[X_i,\vdiv_i]&=-\hdiv_i\\
[X_i,\hdiv_i]&=\vdiv_iR_i\\
\end{split}
\end{equation}
We have denoted $n_i=\dim M_i$.

All differential and curvature operators on~$M_1$ commute with those on~$M_2$ because of the product structure.

On manifolds of dimension one, we can let $\vdiv=\vgrad=\hdiv=\hgrad=0$; this is consistent with all commutator formulas.

\subsection{Pestov identities}

Using the commutator identities given above, we find Pestov identities related to the null X-ray transform.
These are analogous to the Riemannian identities given in~\cite[Section~2]{PSU:anosov-mfld}.

The norms~$\aabs{\cdot}$ below are norms on~$L^2(LM)$, and~$\iip{\cdot}{\cdot}$ is the corresponding inner product.
The norm~$\abs{\cdot}$ may refer to the usual norm on~$T_xM_i$ or~$\R^{n_i}$.

\begin{lemma}
\label{lma:pestov-i}
If $u\colon LM\to\R$ is smooth and vanishes at the boundary, then
\begin{equation}
\aabs{\vgrad_iXu}^2
=
\aabs{X\vgrad_iu}^2
-\iip{R_i\vgrad_iu}{\vgrad_iu}
-(n_i-1)\iip{X_iXu}{u}
\end{equation}
for $i=1,2$.
\end{lemma}

\begin{proof}
Suppose $i=1$; the other proof is identical.
Integrating by parts yields
\begin{equation}
\aabs{\vgrad_1Xu}^2-\aabs{X\vgrad_1u}^2
=
\iip{(X\vdiv_1\vgrad_1X-\vdiv_1XX\vgrad_1)u}{u}.
\end{equation}
Using $X=X_1+X_2$ and commutator formulas, we obtain
\begin{equation}
\begin{split}
&
X\vdiv_1\vgrad_1X-\vdiv_1XX\vgrad_1
\\
&=
(X_1\vdiv_1\vgrad_1X_1-\vdiv_1X_1X_1\vgrad_1)
\\&\quad
+(X_1\vdiv_1\vgrad_1X_2+X_2\vdiv_1\vgrad_1X_1
\\&\qquad-\vdiv_1X_1X_2\vgrad_1-\vdiv_1X_2X_1\vgrad_1)
\\&\quad
+(X_2\vdiv_1\vgrad_1X_2-\vdiv_1X_2X_2\vgrad_1)
\\&=
-(n_1-1)X_1^2+\vdiv_1R_1\vgrad_1
\\&\quad
+X_2([X_1,\vdiv_1]\vgrad_1-\vdiv_1[X_1,\vgrad_1])
\\&\quad
+0
\\&=
-(n_1-1)X_1^2+\vdiv_1R_1\vgrad_1-(n_1-1)X_2X_1.
\end{split}
\end{equation}
Integration by parts in the second term gives the claim.
\end{proof}

\begin{lemma}
\label{lma:pestov-sum}
If $u\colon LM\to\R$ is smooth and vanishes at the boundary, then
\begin{equation}
\begin{split}
&
(n_2-1)\aabs{\vgrad_1Xu}^2
+(n_1-1)\aabs{\vgrad_2Xu}^2
\\&=
(n_2-1)\aabs{X\vgrad_1u}^2
+(n_1-1)\aabs{X\vgrad_2u}^2
\\&\quad
-(n_2-1)\iip{R_1\vgrad_1u}{\vgrad_1u}
-(n_1-1)\iip{R_2\vgrad_2u}{\vgrad_2u}
\\&\quad
+(n_1-1)(n_2-1)\aabs{Xu}^2.
\end{split}
\end{equation}
\end{lemma}

\begin{proof}
The result follows from lemma~\ref{lma:pestov-i}, applied to both $i=1,2$.
\end{proof}

Formally putting $n_2=0$ and $\vgrad_2=0$ in lemma~\ref{lma:pestov-sum} reproduces the Riemannian Pestov identity~\cite[Propostion~2.2]{PSU:anosov-mfld}.

In addition to functions, we will consider one-forms.
We can turn a one-form~$\alpha$ on~$M$ into a function~$f$ on~$LM$ via $f(x,v)=\alpha(x)(v)$.
The one-form is a sum of two one-forms, $\alpha=\alpha_1+\alpha_2$, where $\alpha_i$ is a one-form on~$M_i$.
We remark that~$\alpha_1$ depends on both~$x_1$ and~$x_2$ but only acts on vector fields on~$M_1$; and vice versa for~$\alpha_2$.
To prove injectivity for one-forms, we need to supplement lemma~\ref{lma:pestov-sum} with the following identity.

\begin{lemma}
\label{lma:1-form}
If~$f$ is a smooth function on~$LM$ arising from a one-form, then
\begin{equation}
\label{eq:lma-1-form}
(n_2-1)\aabs{\vgrad_1f}^2
+(n_1-1)\aabs{\vgrad_2f}^2
=
(n_1-1)(n_2-1)\aabs{f}^2.
\end{equation}
\end{lemma}

\begin{proof}
Let us denote the usual measures on the manifolds~$M_i$ and the fibers of the bundles~$SM_i$ by~$\der x_i$ and~$\der v_i$.

Now $\vgrad_if(x_1,x_2,v_1,v_2)$ is the component of $\alpha_i(x_1,x_2)$ orthogonal to~$v_i$, which implies
\begin{equation}
\label{eq:1-form-og}
\abs{\alpha_i}^2
=
\abs{f_i}^2+\abs{\vgrad_if_i}^2
\end{equation}
at every point of~$LM$.

Let~$\Sigma$ be the usual surface measure on the sphere $S^{n-1}\subset\R^n$.
By symmetry the integral $\int_{S^{n-1}}(v_j)^2\der\Sigma(v)$ is independent of the index~$j$, so
\begin{equation}
n\int_{S^{n-1}}(v_1)^2\der\Sigma(v)
=
\sum_{j=1}^n\int_{S^{n-1}}(v_j)^2\der\Sigma(v)
=
\Sigma(S^{n-1}).
\end{equation}
By symmetry and scaling properties, we get for any $a\in\R^n$ that
\begin{equation}
\label{eq:sphere-int}
\int_{S^{n-1}}(a\cdot v)^2\der\Sigma(v)
=
\frac{\abs{a}^2\Sigma(S^{n-1})}{n}.
\end{equation}
Let us denote by $\omega_i=\Sigma(S^{n-1})$ the measure of the unit sphere in~$\R^{n_i}$.

Using these results from the Euclidean space, let us calculate the fiberwise integrals of~$\abs{\vgrad_if_i}^2$ and~$\abs{f_i}^2$.
First, identity~\eqref{eq:sphere-int} gives
\begin{equation}
\label{eq:f_i^2-int}
\int_{S_{x_i}M_i}\abs{f_i(x,v)}^2\der v_i
=
\abs{\alpha_i(x)}^2\frac{\omega_i}{n_i}.
\end{equation}
With equation~\eqref{eq:1-form-og} this yields
\begin{equation}
\label{eq:vgrad-int-i}
\int_{S_{x_i}M_i}\abs{\vgrad_if_i(x,v)}^2\der v_i
=
\abs{\alpha_i(x)}^2\omega_i\frac{n_i-1}{n_i}.
\end{equation}
By symmetry, we have also
\begin{equation}
\label{eq:f_i-int}
\int_{S_{x_i}M_i}f_i(x,v)\der v_i
=
0.
\end{equation}
Combining equations~\eqref{eq:f_i^2-int} and~\eqref{eq:f_i-int} we get for any $x\in M$
\begin{equation}
\label{eq:f^2-int}
\begin{split}
\int_{L_xM_x}\abs{f}^2\der v_1\der v_2
&=
\int_{S_{x_1}M_1}\int_{S_{x_2}M_2}(f_1^2+f_2^2+2f_1f_2)\der v_1\der v_2
\\&=
\int_{S_{x_1}M_1}f_1^2\der v_1\cdot\int_{S_{x_2}M_2}\der v_2
\\&\quad+
\int_{S_{x_1}M_1}\der v_1\cdot\int_{S_{x_2}M_2}f_2^2\der v_2
\\&\quad+
2\int_{S_{x_1}M_1}f_1\der v_1\cdot\int_{S_{x_2}M_2}f_2\der v_2
\\&=
\abs{\alpha_1(x)}^2\frac{\omega_1}{n_1}\cdot\omega_2
+
\omega_1\cdot\abs{\alpha_2(x)}^2\frac{\omega_2}{n_2}
+
0
\\&=
\frac{\omega_1\omega_2}{n_1n_2}(n_2\abs{\alpha_1(x)}^2+n_1\abs{\alpha_2(x)}^2)
\end{split}
\end{equation}
and similarly~\eqref{eq:vgrad-int-i} gives
\begin{equation}
\label{eq:vgrad-int}
\int_{L_xM}\abs{\vgrad_if}^2\der v_1\der v_2
=
\omega_1\omega_2\abs{\alpha_i(x)}^2\frac{n_i-1}{n_i}.
\end{equation}
Integrating equations~\eqref{eq:f^2-int} and~\eqref{eq:vgrad-int} over~$M$ gives~\eqref{eq:lma-1-form}.
\end{proof}

\subsection{Proof of theorem~\ref{thm:product}}

As mentioned in the introduction, it suffices to prove the results for $c=1$ due to conformal invariance of null geodesics.

(1) 
%
%
Let us denote the flight time function $LM\to[0,\infty)$ by~$\tau$.
For $(x,v)\in LM$, let $\gamma_{x,v}\colon[0,\tau_{x,v}]\to M$ be the maximal null geodesic starting at the given point and direction.
We define $u\colon LM\to\R$ by
\begin{equation}
\label{eq:u-def}
u(x,v)
=
\int_0^{\tau_{x,v}}f(\gamma_{x,v}(t))\der t.
\end{equation}
The flight time~$\tau$ does not depend smoothly on $(x,v)\in LM$.
The function~$\tau$ is non-smooth at an interior point~$(x,v)$ of~$LM$ if and only if $\gamma_{x,v}(\tau_{x,v})\in E$.
This does not, however, make~$u$ non-smooth since the function~$f$ vanishes in a neighborhood of~$E$ and we may alter~$\tau$ slightly to make it smooth without changing~$u$.
This is due to the smoothness of the flight time functions $\tau_i\colon SM_i\to[0,\infty)$ on simple Riemannian manifolds.

The fundamental theorem of calculus along each geodesic gives $Xu=-f$.
Since~$f$ does not depend on direction, we have $\vgrad_iXu=0$ for $i=1,2$.
Also,~$u$ vanishes at the boundary of~$LM$ because~$f$ integrates to zero over all null geodesics.
Using lemma~\ref{lma:pestov-sum} for this~$u$ gives $\aabs{Xu}=0$, so $f=0$.
Notice that all terms in the Pestov identity of lemma~\ref{lma:pestov-sum} are non-negative.

(2)
We can identify the one-form~$\alpha$ with a function~$f$ on~$LM$ as discussed before.
We define $u\colon LM\to\R$ as in~\eqref{eq:u-def}, replacing~$f(\gamma_{x,v}(t))$ with $f(\gamma_{x,v}(t),\dot\gamma_{x,v}(t))$.
This function~$u$ is smooth and vanishes at~$\partial(LM)$.

We again apply lemma~\ref{lma:pestov-sum}, now combined with lemma~\ref{lma:1-form}.
We obtain
\begin{equation}
\begin{split}
0
&=
(n_2-1)\aabs{X\vgrad_1u}^2
+(n_1-1)\aabs{X\vgrad_2u}^2
\\&\quad
-(n_2-1)\iip{R_1\vgrad_1u}{\vgrad_1u}
-(n_1-1)\iip{R_2\vgrad_2u}{\vgrad_2u}.
\end{split}
\end{equation}
Since both curvatures were assumed non-positive and both dimensions at least two, this implies that $X\vgrad_iu=0$ for both~$i$.
This in turn means that the vector field~$\vgrad_iu$ is parallel transported along all geodesics, and the boundary values imply that $\vgrad_iu=0$.

Since both vertical gradients of~$u$ vanish and the fibers of~$LM$ are connected,~$u$ can be identified with (pulled back from) a smooth function $-h\colon M\to\R$.
Now $Xu=-f$ together with the identification of~$f$ and~$\alpha$ gives $\alpha=\der h$.
Since~$u$ vanishes at the boundary, so does~$h$.

\subsection{Remarks}

Theorem~\ref{thm:product} can also be regarded as an injectivity result for the X-ray transform on the Riemannian product manifold $M_1\times M_2$ with partial data.
This problem is not conformally invariant.

Assuming at least two dimensions for both space and time ($n_1\geq2$ and $n_2\geq2$) was crucial.
To obtain a useful Pestov identity, we needed to be able to differentiate with respect to the dimension of time -- this was done with~$\vgrad_2$.
In the case of one-forms we also needed the fibers of the light bundle~$LM$ to be connected, which is false if $n_1=1$ or $n_2=1$.


We do not know whether or not there is a useful Pestov identity for general pseudo-Riemannian manifolds, not only conformal pseudo-Riemannian products of Riemannian manifolds.
To achieve such an identity, we should preferably have a compact bundle on which the null geodesic flow lives.
In product geometry we may use the natural Riemannian product metric.
Viewed this way, the manifold in theorem~\ref{thm:product} has a Riemannian metric~$g$ and a pseudo-Riemannian metric~$h$, and the null geodesic flow lives on the bundle
\begin{equation}
LM
=
\{(x,v)\in TM;g(v,v)=1,h(v,v)=0\}.
\end{equation}
For a generic pseudo-Riemannian metric~$h$ we do not expect there to exist a Riemannian metric~$g$ so that the bundle above is is invariant under the null $h$-geodesic flow.

\section{X-ray tomography in Euclidean spaces with partial direction data}
\label{sec:minkowski}

For a set $D\subset S^{n-1}\subset\R^n$, we consider the problem of recovering a function $f\colon\R^n\to\C$ from the knowledge of its line integrals over all lines with direction in the set~$D$.
Our main motivation is in null ray transforms in a Minkowski space of any signature, but we present the method and results in greater generality.

We will consider two function spaces, Schwartz functions ($\schw(\R^n)$) and compactly supported smooth functions ($C_0^\infty(\R^n)$).
Notice that $C_0^\infty(\R^n)\subset\schw(\R^n)$.
The results are likely to hold with lower regularity, but we shall not pursue optimal regularity in this section.

For $f\in\schw(\R^n)$ and $v\in D$, we define the function $R_vf\colon\R^n\to\C$ by
\begin{equation}
R_vf(x)=\int_\R f(x+tv)\der t.
\end{equation}
This function is well-defined and smooth due to the regularity assumption on~$f$, but in general~$R_vf$ is not a Schwartz function.

We denote the normal bundle $N(D)\subset\R^n$ of~$D$ by
\begin{equation}
N(D)=\{x\in\R^n;x\cdot v=0\text{ for some }v\in D\}.
\end{equation}
It is well known~\cite{SU:x-ray-book} that in order to recover a function stably from its integrals over a family of curves, the normal bundle of the family should cover the whole tangent bundle of the space, be it a Euclidean space or a manifold.
The following result is analogous, but due to the structure of the problem at hand, the considerations are global rather than microlocal.

The following theorem is not new, but we record and prove it here for completeness.
At least some aspects of corollary~\ref{cor:x-ray-compact} are also known, but  corolorray~\ref{cor:x-ray-compact}(2b) is new to the best of our knowledge.
The condition that $N(D)=\R^n$ is known as the Orlov condition~\cite{O:orlov-condition}.
For uniqueness we do not need to assume that much.

\begin{theorem}
\label{thm:x-ray-schwartz}
Let $n\geq1$ and fix a set $D\subset S^{n-1}$.
The following are equivalent for a function $f\in\schw(\R^n)$:
\begin{enumerate}
\item
For every $v\in D$ the function~$R_vf$ vanishes identically.
(The function~$f$ integrates to zero over all lines with directions in the set~$D$.)
\item
The Fourier transform~$\F f$ of~$f$ vanishes in~$N(D)$.
\end{enumerate}
In particular, the X-ray transform with directions restricted to the set~$D$ is injective on~$\schw(\R^n)$ if and only if~$N(D)$ is dense in~$\R^n$.
\end{theorem}

\begin{proof}
Take any $v\in D$.
The function~$R_vf$ is invariant under translations in the direction~$v$, so we lose no information in restricting~$R_vf$ to the orthogonal complement $v^\perp\subset\R^n$.
We calculate the Fourier transform of~$R_vf|_{v^\perp}$ and find that
\begin{equation}
\begin{split}
\F_{v^\perp}R_v|_{v^\perp}f(k)
&=
\int_{v^\perp}e^{-ik\cdot x}\int_\R f(x+tv)\der t\der x
\\&=
\int_{\R^n}e^{-ik\cdot y}f(y)\der y
=
\F f(k)
\end{split}
\end{equation}
for all $k\in v^\perp$,
where restriction to the subspace has been indicated by a subscript in the Fourier transform.

Therefore the function~$R_vf$ determines the restriction of~$\F f$ to~$v^\perp$ but gives no other information.
We conclude that each~$R_vf$ vanishes identically for all $v\in D$ if and only if~$\F f$ vanishes on~$\bigcup_{v\in D}v^\perp$.
This union is exactly~$N(D)$.
This concludes the proof of the equivalence of the listed statements.

If~$N(D)$ is dense, then injectivity follows from the obtained result and continuity.
If~$N(D)$ is not dense, there is a non-trivial Schwartz function supported in its complement, and its inverse Fourier transform is in the kernel of the X-ray transform with directions in~$D$.
\end{proof}

If convenient, one may assume that~$D$ is closed and antipodally symmetric without any loss of generality.

In the case of compactly supported functions a smaller set~$D$ suffices for injectivity as the following result shows.
Unfortunately the result does not fully characterize the sets~$D$ for which the ray transform is injective.

\begin{corollary}
\label{cor:x-ray-compact}
Let $n\geq1$ and fix a set $D\subset S^{n-1}$.
Consider the X-ray transform on the space $C_0^\infty(\R^n)$ with directions in the set~$D$.
\begin{enumerate}
\item
If~$D$ is finite, the transform is not injective.
\item
If~$N(D)$ has an interior point, then the transform is injective.
In particular, either of the following suffices for injectivity:
\begin{enumerate}
\item
$D$ is nonempty and open.
\item
$D$ contains a connected set of at least two points.
\end{enumerate}
\end{enumerate}
\end{corollary}

\begin{proof}
(1)
Let $\phi\in C_0^\infty(\R^n)$ be any non-trivial function and write the finite set as $D=\{v_1,v_2,\dots,v_M\}$.
For every $m\in\{1,\dots,M\}$ let $f_m(x)=v_m\cdot\nabla\phi(x)$.
Define then $f\colon\R^n\to\C$ by
\begin{equation}
f
=
f_1*f_2*\cdots*f_M.
\end{equation}
As a convolution of non-trivial~$C_0^\infty$ functions~$f$ is a non-trivial~$C_0^\infty$ function.

The Fourier transform~$\F f$ is the product of the Fourier transforms of the functions $f_1,\dots,f_M$.
Since~$\F f_m$ vanishes on~$v_m^\perp$, we know that~$\F f$ vanishes on the union of all these orthogonal complements.
This union is the normal bundle~$N(D)$, so~$f$ is in the kernel of the X-ray transform by theorem~\ref{thm:x-ray-schwartz}.

(2)
Suppose $f\in C_0^\infty(\R^n)$ satisfies $R_vf=0$ for all $v\in D$.
By theorem~\ref{thm:x-ray-schwartz}~$\F f$ vanishes on~$N(D)$ and in particular on some open set.
The Fourier transform~$\F f$ is real-analytic due to the Paley--Wiener theorem, so~$\F f$ must vanish identically.
Thus $f=0$.

(2a)
If~$D$ is open, it is an elementary observation that~$N(D)$ is open.
This result is also a special case of the next condition since every open set contains an arc.

(2b)
We will show that the unit normal bundle $N_1(D)=N(D)\cap S^{n-1}$ has an interior point as a subset of~$S^{n-1}$.
It follows that~$N(D)$ has an interior point as well.

Let $\gamma\subset D$ be a connected set with at least two points.
Define functions $f_\gamma,g_\gamma\colon S^{n-1}\to[-1,1]$ by
\begin{equation}
\begin{split}
f_\gamma(x)&=\inf_{y\in\gamma} x\cdot y\quad\text{and}\\
g_\gamma(x)&=\sup_{y\in\gamma} x\cdot y.
\end{split}
\end{equation}
These two functions are continuous:
The continuity of~$g_\gamma$ follows from the continuity of the distance function and the property
\begin{equation}
g_\gamma(x)=\cos(d(x,\gamma)).
\end{equation}
If $-\gamma=\{-x;x\in\gamma\}$ denotes the antipodal reflection of~$\gamma$, then $f_\gamma=-g_{-\gamma}$, so also~$f_\gamma$ is continuous.

There exists a point $x_0\in S^{n-1}$ so that $f_\gamma(x_0)<0$ and $g_\gamma(x_0)>0$.
To see this, take two points $a,b\in\gamma$ which are distinct but not antipodal and observe that there is an~$x_0$ for which $a\cdot x_0<0<b\cdot x_0$.

Since the functions~$f_\gamma$ and~$g_\gamma$ are continuous,~$x_0$ has a neighborhood of points with this same property.
It thus suffices to show that $x_0\in N_1(D)$.

The function $\phi\colon S^{n-1}\to[-1,1]$, $\phi(v)=v\cdot x_0$ is continuous.
Therefore the image~$\phi(\gamma)$ is connected, and so $(f_\gamma(x_0),g_\gamma(x_0))\subset\phi(\gamma)$.
Since $f_\gamma(x_0)<0$ and $g_\gamma(x_0)>0$, we have $0\in\phi(\gamma)$.
This means that $x_0\in N_1(D)$ as claimed.
\end{proof}

We now present the conclusions of theorem~\ref{thm:x-ray-schwartz} and corollary~\ref{cor:x-ray-compact} in the context of Minkowski spaces.
Recall that the Minkowski space of signature $(n_1,n_2)$ is the space~$\R^{n_1+n_2}$ with the pseudo-Riemannian metric
\begin{equation}
\der x_1^2+\cdots+\der x_{n_1}^2-\der x_{n_1+1}^2-\cdots-\der x_{n_1+n_2}^2.
\end{equation}
We denote this space by~$\R^{n_1,n_2}$.

A vector $v\in\R^{n_1,n_2}$ in this space is null (or lightlike) if
\begin{equation}
v_1^2+\cdots+v_{n_1}^2-v_{n_1+1}^2-\cdots-v_{n_1+n_2}^2=0.
\end{equation}
The set of all null vectors is called the light cone and denoted here by~$L$.

The null X-ray transform on~$\R^{n_1,n_2}$ is the X-ray transform where integrals are only taken in null directions.
That is, the set of directions~$D$ consists of all unit vectors in the light cone~$L$.

\begin{corollary}
\label{cor:minkowski}
Injectivity of the null X-ray transform on~$\R^{n_1,n_2}$ can be characterized as follows:
\begin{enumerate}
\item
If $n_1=n_2=1$, the null X-ray transform is injective on neither~$\schw(\R^{n_1,n_2})$ nor~$C_0^\infty(\R^{n_1,n_2})$.
\item
If $n_1=1$ and $n_2>1$, the null X-ray transform is injective on~$C_0^\infty(\R^{n_1,n_2})$.
\item
If $n_1=1$ and $n_2>1$, the kernel of the null X-ray transform on~$\schw(\R^{n_1,n_2})$ consists of functions whose Fourier transform is supported in the timelike light cone
\begin{equation}
\{v\in\R^{n_1,n_2};v_1^2-v_{2}^2-\cdots-v_{n_1+n_2}^2\geq0\}.
\end{equation}
\item
If $n_1>1$ and $n_2>1$, the null X-ray transform is injective on both~$\schw(\R^{n_1,n_2})$ and~$C_0^\infty(\R^{n_1,n_2})$.
\end{enumerate}
\end{corollary}

Note that the function spaces and normal bundles are defined with respect to the Euclidean metric on~$\R^{n_1+n_2}$.
The function spaces and the results of the corollary are, however, invariant under Lorentz transformations, so the results have the correct invariance properties.
The result is also invariant under sufficiently regular conformal changes of the metric.

\begin{proof}[Proof of corollary~\ref{cor:minkowski}]
(1)
This follows from corollary~\ref{cor:x-ray-compact}(1) since~$D$ consists of four points.

(2)
This follows from corollary~\ref{cor:x-ray-compact}(2b) since~$D$ contains an arc.

(3)
This follows from theorem~\ref{thm:x-ray-schwartz} since the normal bundle of the ``unit light cone''~$L\cap S^{n_1+n_2-1}$ is precisely the closure of the complement of the timelike light cone.

(4)
This follows from theorem~\ref{thm:x-ray-schwartz} since now $N(L\cap S^{n_1+n_2-1})=\R^{n_1,n_2}$.
\end{proof}

\section{Null X-ray tomography on conformal Minkowski tori}
\label{sec:torus}

We quotient the Minkowski space~$\R^{n_1,n_2}$ by the lattice~$\Z^{n_1+n_2}$.
Since the metric on~$\R^{n_1,n_2}$ is translation invariant, we obtain a pseudo-Riemannian quotient manifold $\T^{n_1,n_2}\coloneqq\R^{n_1,n_2}/\Z^{n_1+n_2}$ which we call the Minkowski torus of signature $(n_1,n_2)$.

We point out that the lattice is not Lorentz invariant, so Minkowski tori have less symmetry than Minkowski spaces.
A conformal Minkowski torus is a Minkowski torus with the metric multiplied with a smooth conformal factor.

The Minkowski tori are closed manifolds, so maximal geodesics are closed or have infinite length.
We only include the closed ones.
The null periodic X-ray transform takes a function on the torus into its integrals over all periodic null geodesics.
This transform is injective if and only if a function can be recovered from the knowledge of these integrals.

\begin{theorem}
\label{thm:torus}
The null periodic X-ray transform is injective on a conformal Minkowski torus of signature $(n_1,n_2)$ if and only if $n_1\geq2$ and $n_2\geq2$.
If it is injective, it is injective on distributions.
If it is not injective, it is non-injective on the space of trigonometric polynomials.
\end{theorem}

\begin{theorem}
\label{thm:torus-ker}
Let $n_1\geq1$.
The kernel of the null periodic X-ray transform on a conformal Minkowski torus of signature $(n_1,1)$ is precisely the set of functions whose Fourier series is supported outside the set
\begin{equation}
\label{eq:K-def}
\begin{split}
K
&=
\{(k,t)\in\Z^{n_1}\times\Z;\exists(v,s)\in\Z^{n_1}\times\Z:
\\&\qquad
\abs{v}^2=s^2\neq0,v\cdot k+st=0\}.
\end{split}
\end{equation}
We always have
\begin{equation}
\label{eq:K-est}
\{(k,t)\in\Z^{n_1}\times\Z;\abs{t}=\abs{k}\}
\subset
K
\subset
\{(k,t)\in\Z^{n_1}\times\Z;\abs{t}\leq\abs{k}\}
.
\end{equation}
For $n_1=1$ we have
\begin{equation}
\label{eq:K-eq}
K
=
\{(k,t)\in\Z\times\Z;\abs{t}=\abs{k}\},
\end{equation}
for $n_1=2$
\begin{equation}
\label{eq:K-eq2}
K
=
\{(k,t)\in\Z^2\times\Z;\sqrt{\abs{k}^2-\abs{t}^2}\in\N\},
\end{equation}
and for $n_1=3$
\begin{equation}
\label{eq:K-eq3}
K
=
\{(k,t)\in\Z^3\times\Z;\abs{k}^2-\abs{t}^2\text{ is a sum of two integer squares}\}.
\end{equation}
\end{theorem}

A well known theorem (see eg.~\cite[Theorem~366]{HW:number-theory}) states that a positive integer is a sum of two integer squares if and only if every prime factor which is~$3$ modulo~$4$ appears an even number of times.

For a precise definition of what it means to integrate a distribution over a closed geodesic on a torus, see~\cite{I:torus}.

Let us compare corollary~\ref{cor:minkowski}(3) and theorem~\ref{thm:torus-ker}.
They both state roughly that if there is only one dimension of time ($n_1=1$ or $n_2=1$), then having vanishing null X-ray transform corresponds to having the Fourier transform supported in the timelike light cone.
In a Minkowski space this is true as stated, but on a Minkowski torus there is an additional condition that a Diophantine system must be solvable.
Our observations suggest that the Diophantine system is always solvable if $n_1\geq5$.
We will discuss this condition more closely in section~\ref{sec:diophantos}.

\subsection{Proofs of theorems~\ref{thm:torus} and~\ref{thm:torus-ker}}

A conformal change of metric only reparametrizes null geodesics, so we may assume the metric to be the standard Minkowski metric.

Let us define
\begin{equation}
\begin{split}
K
&=
\{(k,p)\in\Z^{n_1}\times\Z^{n_2};\exists(v,w)\in\Z^{n_1}\times\Z^{n_2}:
\\&\qquad
\abs{v}^2=\abs{w}^2\neq0,v\cdot k+w\cdot p=0\}.
\end{split}
\end{equation}
If $n_2=1$, this agrees with~\eqref{eq:K-def}.

Our problem is to recover $f\colon\T^{n_1,n_2}\to\C$ from the knowledge of the function
\begin{equation}
R_vf(x)
=
\int_0^1 f(x+tv)\der t
\end{equation}
for all non-zero null vectors $v\in\Z^{n_1+n_2}$.

It follows from~\cite[Equation~(6) and Lemma~10]{I:torus} that the Fourier coefficient~$\hat f(k,p)$ can be recovered from this data if and only if $(k,p)\in K$.
Specifically,
\begin{equation}
\widehat{R_vf}(k)
=
\begin{cases}
\hat f(k) & \text{if }v\cdot k=0\\
0 & \text{if }v\cdot k\neq0.
\end{cases}
\end{equation}
That is, the kernel of the null X-ray transform on~$\T^{n_1,n_2}$ is precisely the set of functions whose Fourier series vanishes on~$K$.
This is true also for distributions.

Let us show that if $n_1,n_2\geq2$, then $K=\Z^{n_1+n_2}$.
Let 
\begin{equation}
k=((k_{11},\dots,k_{1n_1}),(k_{21},\dots,k_{2n_2}))\in\Z^{n_1}\times\Z^{n_2}.
\end{equation}
It suffices to show that $(k_{11},k_{12};k_{21},k_{22})$ is in $K\subset\Z^{2+2}$.
By symmetries we may assume that all components are positive and and $k_{i1}<k_{i2}$ for both $i=1,2$.
We will show that $v=(a,-b,b,-a)$ is the desired vector for suitable $a,b\in\Z$.
We need to have
\begin{equation}
a(k_{11}-k_{22})-b(k_{12}-k_{21})=0.
\end{equation}
An example is $a=k_{12}-k_{21}$ and $b=k_{11}-k_{22}$.
If this leads to $a=b=0$, then $k=(c,d;d,c)$ and we may choose $a=d$ and $b=c$.
If this leads to $v=0$ as well, then $k=0$ and we can take any~$v$.
Thus indeed $k\in K$ and so $K=\Z^{n_1+n_2}$.

It follows from~\eqref{eq:K-est} that $K\neq\Z^{n_1+n_2}$ if $n_1=1$ or $n_2=1$.
This leads to non-injectivity; one can simply take the inverse Fourier transform of any non-trivial function vanishing in~$K$.

Let us then prove~\eqref{eq:K-est}.
Suppose $n_2=1$, $k_2\neq0$ and $k=(k_1,k_2)\in K\subset\Z^{n_1}\times\Z$.
Then $v=(v_1,v_2)$ satisfies $v_2=ak_2$ for some $a\in\Q$ and $-a\abs{k_2}^2=k_1\cdot v_1$.
The last equation with $\abs{v_1}=\abs{v_2}$ yields $\abs{a}\cdot\abs{k_2}^2\leq\abs{k_1}\cdot\abs{a}\cdot\abs{k_2}$.
This implies $\abs{k_2}\leq\abs{k_1}$.

If $\abs{k_1}=\abs{k_2}$, then $k\in K$; one can choose $v=k_1$ and $s=-k_2$.
To see~\eqref{eq:K-eq}, one needs~\eqref{eq:K-est} and the symmetry in interchanging~$n_1$ and~$n_2$.
A characterization of the set~$K$ is much harder when $n_1>1$.
For $n_1=2$ this is done in lemma~\ref{lma:2d-diofantos} and for $n_1=3$ in lemma~\ref{lma:3d-diofantos} below.

These observations prove theorems~\ref{thm:torus} and~\ref{thm:torus-ker}.

\subsection{A Diophantine problem}
\label{sec:diophantos}

Characterizing the kernel of the null periodic X-ray transform in signature $(n,1)$ requires finding a characterization of the set~$K$.
This boils down to analyzing the dependence of the solvability of a Diophantine system on integer parameters.
We do that in this section for $n=2$ and $n=3$.

\begin{lemma}
\label{lma:2d-diofantos}
Let $k_1,k_2,t\in\Z$.
There exists a solution $(v_1,v_2,s)\in\Z^3$ to
\begin{equation}
\label{eq:2d-diofantos}
\begin{cases}
v_1^2+v_2^2=s^2\neq0\\
k_1v_1+k_2v_2+ts=0
\end{cases}
\end{equation}
if and only if $\sqrt{k_1^2+k_2^2-t^2}$ is integer.
\end{lemma}

\begin{proof}
Suppose $k_1,k_2,t\in\Z$ are such that there are $v_1,v_2,s\in\Z$ solving~\eqref{eq:2d-diofantos}.
Existence of solutions is clearly independent of signs of~$k_1$, $k_2$ and~$t$.
Therefore we are free to assume that $v_1,v_2,s\geq0$.

By scaling $(v_1,v_2,s)$ we may assume that $\gcd(v_1,v_2,s)=1$.
Since $(v_1,v_2,s)$ is then a primitive Pythagorean triple, there are integers $m\geq n\geq0$ so that $s=m^2+n^2$, $v_1=m^2-n^2$ and $v_2=2mn$ (and $m\geq1$).
Let us write $\alpha=n/m$.

The condition $k_1v_1+k_2v_2+ts=0$ becomes
\begin{equation}
\label{eq:diofantos-alpha}
(t-k_1)\alpha^2+2k_2\alpha+(t+k_1)=0
.
\end{equation}
If $t=k_1$, we have $\alpha=-k_1/k_2$.
If $t\neq k_1$, we have
\begin{equation}
\alpha
=
\frac{-k_2\pm\sqrt{k_1^2+k_2^2-t^2}}{t-k_1}
.
\end{equation}
This~$\alpha$ must be rational, so $\sqrt{k_1^2+k_2^2-t^2}$ must be an integer.

Conversely, if $\sqrt{k_1^2+k_2^2-t^2}$ is an integer (which is always true if $t=k_1$), then one can find a rational solution~$\alpha$ to~\eqref{eq:diofantos-alpha}.
This~$\alpha$ gives rise to some integers~$m$ and~$n$ so that $\alpha=m/n$, and these in turn produce a solution to the original system~\eqref{eq:2d-diofantos}.
\end{proof}

\begin{lemma}
\label{lma:3d-diofantos}
Let $k_1,k_2,t\in\Z$.
There exists a solution $(v_1,v_2,v_3,s)\in\Z^4$ to
\begin{equation}
\label{eq:3d-diofantos}
\begin{cases}
v_1^2+v_2^2+v_3^2=s^2\neq0\\
k_1v_1+k_2v_2+k_3v_3+ts=0
\end{cases}
\end{equation}
if and only if $k_1^2+k_2^2+k_3^2-t^2$ is a sum of two integer squares.
\end{lemma}

\begin{proof}
Let us denote $D=k_1^2+k_2^2+k_3^2-t^2$.
Similarly to the previous proof, we may use a parametrization of Pythagorean quadruples.
Suppose $(v_1,v_2,v_3,s)$ is a solution consisting of non-negative numbers with $\gcd(v_1,v_2,v_3,s)=1$.

There is a factor~$p$ of $v_1^2+v_2^2$ so that
\begin{equation}
\begin{split}
v_3&=\frac1{2p}(v_1^2+v_2^2-p^2)\quad\text{and}\\
s&=\frac1{2p}(v_1^2+v_2^2+p^2).
\end{split}
\end{equation}
To see this, simply take $p=s-v_3$.

Let $\alpha=v_1/p$ and $\beta=v_2/p$.
The condition $k_1v_1+k_2v_2+k_3v_3+ts=0$ then becomes
\begin{equation}
2k_1\alpha+2k_2\beta+k_3(\alpha^2+\beta^2-1)+t(\alpha^2+\beta^2+1)=0.
\end{equation}
We consider separately the cases $k_3+t=0$ and $k_3+t\neq0$.

If $k_3+t=0$, we have
\begin{equation}
2k_1\alpha+2k_2\beta+2k_3=0.
\end{equation}
This always has rational solutions~$\alpha$ and~$\beta$.
Notice that if $k_3+t=0$, then $D=k_1^2+k_2^2$, a sum of two squares.

If $k_3+t\neq0$, we denote $x=(k_3+t)\alpha+k_1$ and $y=(k_3+t)\beta+k_2$.
Our equation becomes simply $x^2+y^2=D$.
Let $c\in\N$ be such that $cx,cy\in\Z$.
Now $(cx)^2+(cy)^2=c^2D$.

A positive integer is a sum of two integer squares if and only if every prime factor which is~$3$ modulo~$4$ appears an even number of times.
Since~$c^2$ contains each prime an even number of times, it follows that $x^2+y^2=D$ has rational solutions~$(x,y)$ if and only if it has integer solutions~$(x,y)$.

The argument can be reversed.
The special case $k_3+t=0$ is simple, so let us ignore it.
If~$D$ is a sum of two squares, we find integers~$x$ and~$y$ so that $x^y+y^2=D$.
From these integers we calculate the rationals~$\alpha$ and~$\beta$.
Assuming $p=1$, we find a rational solution $(v_1,v_2,v_3,s)$ to our original pair of equations.
Scaling this by a suitable integer gives an integer solution.
\end{proof}

We have shown that for $n\in\{1,2,3\}$ the there exists a solution $(v,s)\in\Z^n\times\Z$ to
\begin{equation}
\label{eq:nd-diofantos}
\begin{cases}
\abs{v}^2=s^2\neq0\\
k\cdot v+ts=0
\end{cases}
\end{equation}
if and only if $\abs{k}^2-t^2$ is a sum of $n-1$ squares of integers.
We do not know whether or not this is true for $n\geq4$.

We do not pursue the study of Dipohantine equations further here, but we mention two related theorems.
By a theorem of Legendre a natural number can be expressed as the sum of three integer squares if and only if it is not of the form $4^x(8y+7)$ for $x,y\in\N$.
By a theorem of Lagrange every natural number is expressible as the sum of four integer squares.
This \emph{suggests} but does not prove a simple characterization of the kernel for $n_1\geq4$ when $n_2=1$.

\section*{Acknowledgements}

Some of this work was completed during the author's visit to the University of Washington, Seattle, and he is indebted for the hospitality and support offered there; this inculdes financial support from Gunther Uhlmann's NSF grant DMS-1265958.
The author was also partially supported by an ERC Starting Grant (grant agreement no 307023).
He is also grateful to Gunther Uhlmann, Todd Quinto and Mikko Salo for discussions.

\bibliographystyle{abbrv}
\bibliography{ip}

\end{document}